\pdfoutput=1

\documentclass[12pt]{amsart}

\usepackage[backref]{hyperref}

\usepackage{bm}
\usepackage{amsthm}
\usepackage{graphicx}

\usepackage{amssymb}
\usepackage{faktor}

\usepackage{tikz-cd}

\usepackage{mathrsfs}
\usepackage{latexsym}
\usepackage{color}
\usepackage[retainorgcmds]{IEEEtrantools}

\usepackage{tikz, wrapfig,kantlipsum}
\usepackage{subcaption}
\usepackage{float}

\usetikzlibrary{positioning, shapes.geometric}

\usepackage{todonotes}
\usepackage[letterpaper,margin=1.0in,top=1.15in,bottom=1.15in]{geometry}

\newtheorem{theorem}{Theorem}[section] 
\newtheorem{theoremx}{Theorem}
\newtheorem{lemma}[theorem]{Lemma}
\newtheorem{corollary}[theorem]{Corollary}
\newtheorem{prop}[theorem]{Proposition}

\newtheorem{rem}[theorem]{Remark}

\newtheorem{ex}[theorem]{Example}
\newtheorem{defn}[theorem]{Definition}

\title[Geometric Bordisms of the Accola-Maclachlan, Kulkarni and Wiman Surfaces]{Geometric Bordisms of the Accola-Maclachlan, Kulkarni and Wiman Type II Surfaces}
\author{L. Ferrari}

\address{\newline Institute of Mathematics
\newline Polish Academy of Sciences
\newline  Sniadeckich 8
\newline 00-656 Warsaw, Poland}
\email{leonardocpferrari@gmail.com}

\thanks{L.F. was partially supported by the Polish Center of Science Grant OPUS 2019/35/B/ST1/01120.}

\begin{document}

\begin{abstract} In this paper, we prove that the Accola-Maclachlan surface of genus $g$ bounds geometrically an orientable compact hyperbolic $3$-manifold for every genus. For infinitely many genera, these surfaces are also explicit examples of non-arithmetic surfaces that embed geodesically into a closed non-arithmetic manifold. We also provide explicit geodesic embeddings to the Wiman type II and Kulkarni surfaces of every genus, and prove that these surfaces bound geometrically a compact, orientable manifold for $g\equiv 1 \,( \text{mod} \, 2)$ or $g \equiv 3 \,( \text{mod} \, 8)$, respectively.
\end{abstract}

\maketitle

\section{Introduction}
\label{intro}

Bordism properties of closed manifolds have been a classical and important topic in topology. Whilst classical results concern topological or smooth boundaries, in \cite{LR0} the question of  {\em bounding geometrically} was introduced: namely whether a connected closed orientable hyperbolic $n$--manifold $M$ could arise as the totally geodesic boundary of a compact hyperbolic $(n+1)$--manifold $W$. This problem is particularly complicated in dimension $2$, given that there are uncountably many non-isometric surfaces for every given genus. 
It follows from the Mostow's rigidity theorem,
moreover, that the subset of surfaces of a given genus $g$ that do bound is countable. However, 
so far only a few \textbf{explicit} examples of surfaces that do bound geometrically have been constructed.

In \cite{KRS} another problem was considered: whether a given connected closed orientable hyperbolic $n$--manifold $M$ could {\em embed geodesically}, that is arise as an embedded totally geodesic codimension $1$ submanifold of a hyperbolic $(n+1)$--manifold $W$. The authors, in fact, do show that every arithmetic, compact, hyperbolic manifold of even dimension embeds geodesically; and while these two problems are connected (see Section \ref{embedding}), the methods in \cite{KRS} do not apply to non-arithmetic hyperbolic manifolds or provides explicit examples of the embeddings themselves.\footnote{In \cite{KRiS}, these methods were also extended to some non-arithmetic manifolds, namely those that can be obtained by gluing arithmetic pieces.} 


In this paper, we show an explicit construction of geometric bounding for a famous family of Riemann surfaces, the \textit{Accola-Maclachlan surfaces} $X_g$ of genus $g\ge 2$. They are the surfaces described by the algebraic curve equation

\begin{equation}
y^2=x^{2g+2}-1.
\end{equation}

These surfaces were shown independently by Accola \cite{Accola} and Maclachlan \cite{Maclachlan} to have full conformal automorphism group

\begin{equation}\label{AM group}
\mathrm{Aut}(X_g)=AM_g=\langle a,b \mid a^{2g+2}=b^4=(ab)^2=[a,b^2]=1\rangle \cong (\mathbb{Z}_{2g+2}\times \mathbb{Z}_2) \rtimes \mathbb{Z}_2,
\end{equation}
a group of order $8g+8$. Later, Kulkarni \cite{Kulkarni 1} also showed that, for $g\not\equiv 3 (\text{mod}\, 4)$, this is the only surface with a full automorphism group of such order, except for a finite number of genera;\footnote{From \cite[Proposition 4.2]{Kulkarni 2} and \cite[Sections 4,5]{Kulkarni 1}, one can calculate that the exceptional genera are precisely $g=2,3,4,5,9,11,13,14,17,27,29,39,65,89,125,189$.} but the only such group $G$ with this order and action signature $(0;2,4,2g+2)$, meaning that the quotient $S_g/_G$ is an orbifold of such signature. 

For $g\equiv 3(\text{mod}\, 4)$ Kulkarni also showed there's a second surface, now called \textit{Kulkarni surface $Y_g$} of genus $g$, with a conformal automorphism group $K_g$ of order $8g+8$, with group presentation 
\begin{equation}\label{K group}
\mathrm{Aut}(Y_g)=K_g=\langle a,b \mid a^{2g+2}=b^4=(ab)^2=1,b^2ab^2=a^{g+2}\rangle.
\end{equation}
He also showed that $AM_g$ and $K_g$ are always distinct and the only two such conformal automorphism groups with order $8g+8$ for any genus excluding the exceptional genera, but still the only such groups with signature $(0;2,4,2g+2)$ in any genus $g\equiv 3(\text{mod}\, 4)$. 

From the classification of extendability of Fuchsian signatures by Singerman \cite[Theorems 1, 2]{Singerman}, it follows that these two group are also maximal for any $g\neq 3$; while in $g=3$ it follows from \cite[Case T11]{BCC} that in $AM_3$ is also maximal while $K_3$ is extendable to a group of order $96$, which is unique among the non-hyperelliptic genus $3$ curves \cite{KK}. Since the Kulkarni surfaces are non-hyperelliptic \cite[Equation 2]{Turbek}, we have that the Accola-Maclachlan and the Kulkarni surfaces are entirely determined their conformal automorphism groups containing $AM_g$ or $K_g$, and by these subgroups acting with signature $(0;2,4,2g+2)$. 

Using this fact, we will prove the following for the Accola-Maclachlan surfaces:

\begin{theoremx} \label{main}
The Accola-Maclachlan surface $X_g$ of genus $g$ bounds geometrically a compact, orientable, hyperbolic $3$-manifold $M_g$ for every $g\ge 2$.
\end{theoremx}

We note that, since the surfaces $X_g$ are commensurable with the triangle groups $(2,4,2g+2)$, they are arithmetic only for $g=2,3,4,5,8$ (with classes II, III, VI, V, XII, respectively, by Takeuchi's classification \cite{Takeuchi}). 
Moreover, from the proof of Corollary \ref{AM embeds}, the manifolds $M_g$ are commensurable with $R(2g+2)$, the right-angled L\"obell polytope with two $(2g+2)$-gons. More precisely, $M_g$ are obtained by the face-pairing of $8$ copies of $R(2g+2)$, and thus have volume\footnote{The actual volume can be calculated numerically for every $g$ by \cite[Theorem 3.19]{Vesnin}.} asymptotically equivalent to $20v_3(g+1)$ as $g\rightarrow\infty$ \cite[Corollary 3.18]{Vesnin}, where $v_3 \sim 1.0149$ is the volume of the ideal regular hyperbolic tetrahedron. Also, by \cite[Lemma 3.8]{Vesnin}, the reflection groups $\Gamma\big(R(2g+2)\big)$ are only arithmetic in the cases $g=2,3$.\footnote{\cite{KRS} also provides lattice embeddings in this  case.} We then also have the following:

\begin{theoremx}
For $g>8$, the surfaces $X_g$ are explicit examples of non-arithmetic surfaces that embed geodesically into compact non-arithmetic orientable hyperbolic $3$-manifolds and bound geometrically compact orientable hyperbolic $3$-manifolds.
\end{theoremx}

We will also prove the following for the Kulkarni surfaces:

\begin{theoremx}
The Kulkarni surface $Y_g$ of genus $g$ embeds geodesically for every $g\equiv 3(\text{mod}\, 4)$, and bounds geometrically an orientable manifold for every odd $g\equiv 3(\text{mod}\, 8)$.
\end{theoremx}

We note that, from the proof of Corollary \ref{Kulkarni embeds}, the geodesic embedding and the geometric bounding are both into compact, orientable manifolds obtained from the face-pairings of $8$ copies of the L\"obell polytopes $R(2g+2)$, from which the same comments on volume and arithmeticity follow as above.

We will finally show an explicit geodesic embedding, and geometric bounding for odd genus, for another famous family of surfaces, the \textit{Wiman type II surfaces $Z_g$} of genus $g\ge 2$. These surfaces were described by Wiman \cite{Wiman} by the algebraic curve equation

\begin{equation}
    y^2 = x^{2g+1}-x,
\end{equation}
and they are entirely classified as the unique surface up to isometry admitting a conformal automorphism of order $4g$ \cite[Theorem 2]{Nakagawa} for any $g\neq 3$. In the case $g=3$, there is a second surface with an automorphism of order $4g=12$, but the Wiman type II is the only surface with a cyclic conformal automorphism subgroup of order $4g$ acting with signature $(0;2,4g,4g)$ \cite[p. 439]{Nakagawa}.

Using this characterization, we will prove the following:

\begin{theoremx} 
The Wiman type II surface $Z_g$ of genus $g$ embeds geodesically for every $g\ge 2$, and bounds geometrically an orientable manifold $N_g$ for every odd $g\ge 2$.
\end{theoremx}

We also note that, since the surfaces $Z_g$ are commensurable with the triangle groups $(2,4,4g)$, they are arithmetic only for $g=2,3$ (with classes III, V, respectively \cite{Takeuchi}). Moreover, from the proof of Corollary \ref{Wiman embeds}, the manifolds $N_g$ are commensurable with $R(4g)$, obtained by the face-pairing of $4$ copies of the polyhedron, and have volume asymptotically equivalent to $20v_3 g$. In addition, the reflection groups $\Gamma\big(R(4g)\big)$ are only arithmetic in the case $g=2$. 




The constructions in Theorems \ref{is AM}, \ref{is Wiman}, \ref{is Kulkarni} will also give the following characterization of the above surfaces:

\begin{theoremx}
The Accola-Maclachlan, Wiman type II and Kulkarni surfaces of genus $g$ can be obtained by the side-pairing of $4$ right-angled $(2g+2)$-gons, $2$ right-angled $4g$-gons and $4$ right-angled $(2g+2)$-gons, respectively. Labelling the edges of the polygons in clockwise order with numbers $1,2,...$ and the copies of polygons by vectors in $\mathbb{Z}_2 ^2$ ($\mathbb{Z}_2$, $\mathbb{Z}_2 ^2$), we have the following description of the orientation-preserving edge pairings:

\begin{enumerate}
    \item[AM] $(i,v)$ is paired with $(i,v+e_1)$ if $i$ is odd; and with $(i,v+e_2)$ if $i$ is even.
    \item[W2] $(i,0)$ is paired with $(i,1)$ if $i$ is odd; and with $(2g+i,1)$ if $i$ is even.
    \item[K] $(i,v)$ is paired with $(i,v+e_1)$ if $i \equiv 1 (\, \text{mod}\, 4)$; with $(i,v+e_2)$ if $i \equiv 2 (\, \text{mod}\, 4)$; with $(g+1+i,v+e_1)$ if $i \equiv 3 (\, \text{mod}\, 4)$; and with $(g+1+i,v+e_2)$ if $i \equiv 0 (\, \text{mod}\, 4)$,
\end{enumerate}
where the edge labels $N=\{1,2,...\}$ are seen in $\mathbb{Z} \,(\text{mod} \, |N|)$. 
\end{theoremx}

We note that this characterization is not unique. \cite[Theorem 6.3]{Weaver}, for instance, characterizes these families and others by face-pairings of a single suitable regular polygon with suitable inner angles; and the Wiman type II curve in genus $g=2$, also known as the \textit{Bolza curve}, has an additional characterization as side-pairings of $6$ regular octagons with inner angles $\frac{2\pi}{3}$. The above characterization, however, has the advantage of using only right-angled polygons.



We'd like to stress, however, that this paper does not answer the question posed by Zimmermann \cite{Zimmermann} and \cite{KRS} regarding the geometric bounding of the \textit{Klein quartic}, the surface which realizes the Hurwitz bound $|\mathrm{Aut}(S_g)|\le 84(g-1)$ in the lowest possible genus $g=3$. In fact, having algebraic curve equation $y^3+yx^3+x=0$ \cite{Levy}, the Klein Quartic does not belong to  the above families. It also does not admit fixed-point free, orientation-reversing involutions \cite{Singerman 2}.

\subsection{Structure of the paper}

Section \ref{embedding} defines geodesic embeddings and how to promote them to geometric boundings. Section \ref{colouring} recalls some useful basic results from the theory of colourings of right-angled polytopes. Section \ref{AM} describes the explicit construction of geometric boundings for the Accola-Maclachlan surfaces $X_g$ of every genus. Section \ref{Wiman} described the construction for the Wiman surfaces $Z_g$. Finally, Section \ref{Kulkarni} describes the construction for the the Kulkarni surfaces $Y_g$. The sections were organized by the complexity of the constructions and each uses results from the previous ones.

\medskip

\noindent{\bf Acknowledgements:}~{\normalfont The author is grateful to Alan Reid for the insightful talks and comments on the paper; to Benson Farb for the suggestions and references that allowed this work; to Tadeusz Januszkiewicz for comments and suggestions on an earlier version of the paper; and to Leone Slavich and Stefano Riolo for corrections and fruitful discussions. 
}

\section{Geodesic embeddings and geometric boundaries}\label{embedding}

Throughout this article, we will consider only manifolds that are closed, orientable, hyperbolic and complete, unless explicitly stated.

We say that a manifold $N^n$ \textit{embeds geodesically} if there is an embedding $i:N\to M^{n+1}$ such that $i(N)$ is \textit{totally geodesic} in $M$, that is, if the geodesics in $(M,g)$ are still geodesic in $(i(N),g)$.

We say that a manifold $N^n$ \textit{bounds geometrically} if there is a compact manifold $M^{n+1}$ with totally geodesic boundary such that $\partial M$ is isometric to $N$.

Clearly, bounding geometrically is equivalent to embedding geodesically as a separating totally geodesic submanifold: from the embedding you can cut along to produce two compact components with a connected totally geodesic boundary; and from the geometric bounding you can double the manifold along the boundary to produce a geodesic embedding.

A non-separating embedding, however, is not enough to produce a geometric boundary: if $N$ embeds geodesically in $M$ in a non-separating way, then by cutting $M$ along $N$ we obtain a manifold $M'$ with boundary components $N^+ \sqcup N^-$, where each component is given opposite orientation.

In some cases, nevertheless, it is still possible to promote a geodesic embedding to a geometric bounding. We recall the following \cite[Lemma 3.1]{FKR}:

\begin{lemma}\label{lemma-embed-bound}
Let $N$ be an orientable hyperbolic $n$-manifold that has a fixed point free involution $\varphi \in \mathrm{Isom}(N)$. If $N$ embeds geodesically then it also bounds geometrically.
\end{lemma}

\begin{proof}
Let $N$ embed into an orientable manifold $M$ as a totally geodesic submanifold of codimension $1$. Denote by $M'$ the manifold obtained by cutting $M$ along $N$ and taking a connected component. Then either $\partial M' = N$ and we are done, or $\partial M' = N \sqcup N$, and we can quotient out one copy of $M$ in $\partial N$ by self-identifying it via $\varphi$. Given that $\varphi$ is a fixed point free involution, the resulting metric space $M'_\varphi$ will be a hyperbolic manifold with a single boundary component isometric to $N$. Moreover, $M'_\varphi$ is orientable or not depending on whether $\varphi$ is orientation-reversing or not. 
\end{proof}

In this article, we will also need the following useful fact concerning the immersion of isometry groups of quotient manifolds, specially in the case that the action group $H$ is generated by the fixed-point free involution:

\begin{lemma}\label{quotient lemma}
Let $M$ be a hyperbolic manifold and $H < \mathrm{Isom}(M)$ such that $H$ acts freely. Then $M/_H$ is also a hyperbolic manifold and $$N_{\mathrm{Isom}(M)}(H)/_H < \mathrm{Isom}\big(M/_H\big).$$ If $M$ is orientable, the statement is also true when replacing $\mathrm{Isom}$ with $\mathrm{Isom}^+$, and in this case it also follows that $M/_H$ is orientable.
\end{lemma}

\begin{proof}
Since $M$ is hyperbolic, $\mathrm{Isom}(M)$ is finite, from which it follows that $M\mapsto M/_H$ is a regular cover for any $H<\mathrm{Isom}(M)$ acting freely. 

Take one such $H$ and any $\varphi \in \mathrm{Isom}(M)$. If $\varphi H = H \varphi$, we can naturally define a map $\tilde{\varphi}:M/_H \to M/_H$ by putting $\tilde{\varphi}([x])=[\varphi(x)]$, and this is an isometry of $M/_H$. We thus have a map $\Psi:N_{\mathrm{Isom}(M)}(H) \to \mathrm{Isom}\big(M/_H\big)$ which maps $\varphi$ to $\tilde{\varphi}$, and whose kernel is precisely $H$. It follows that there is an injection $\Psi/_{\text{ker}\, \Psi}$ from $N_{\mathrm{Isom}(M)}(H)/_H$ into $\mathrm{Isom}\big(M/_H\big)$.
\end{proof}

\section{Colourings}\label{colouring}

We now introduce a very useful technique for the construction of explicit examples.

\subsection{Definitions}

A compact hyperbolic polytope $\mathcal{P} \subset \mathbb{H}^n$ is called {\itshape right--angled} if any two codimension $1$ faces (or facets, for short) are either intersecting at a right angle or disjoint. 


Let $\mathcal{P} \subset \mathbb{X}^n$ be a compact, right-angled polytope with the set of facets $\mathcal{F}$. A \textit{colouring} of $\mathcal{P}$ is a map $\lambda: \mathcal{F} \rightarrow V$, where $V$ an $\mathbb{Z}/2 \mathbb{Z}$--vector space. The map $\lambda$ is called \textit{proper} if, for every vertex $v=F_1\cap \ldots \cap F_n$, the vectors $\lambda(F_1), \ldots, \lambda(F_n)$ are linearly independent.

The \textit{rank} of $\lambda$ is the $\mathbb{Z}/2\mathbb{Z}$--dimension of $\mathrm{im}\, \lambda$. We will always assume in this article that colourings are surjective, in the sense that the image of the map $\lambda$ is a generating set of vectors for $V$.

A colouring of a right-angled $n$--polytope $\mathcal{P}$ naturally defines a homomorphism, which we still denote by $\lambda$ without much ambiguity, from the associated right--angled 
reflection group~$\Gamma(\mathcal{P})$ 
into $V$ with its natural group structure. Being a Coxeter polytope, $\mathcal{P}$ has a natural orbifold structure as the quotient $\mathbb{H}^n /_{\Gamma(\mathcal{P)}}$.

\begin{prop}[\cite{DJ91}, Proposition 1.7]\label{DJ}
If the colouring $\lambda$ is proper, then $\ker \lambda < \Gamma(\mathcal{P})$ is torsion-free, and $\mathcal{M}_\lambda = \mathbb{H}^n /_{\ker \lambda}$ is a closed hyperbolic manifold. Moreover, $\mathcal{M}_\lambda \to \mathcal{P}$ is an orbifold covering of degree $|V|=2^{\text{rk} \, \lambda}$.
\end{prop}

\begin{rem}\label{orbifold definition}
$\mathcal{M}_\lambda$ is homeomorphic to the topological manifold $(\mathcal{P} \times \mathbb{Z}_2 ^m)/_{\sim}$, where
\begin{equation}\label{colouring 2}
    (p,g) \sim (q,h) \Leftrightarrow p=q \text{ and } g - h \in G_f 
\end{equation}
$f= F_{i_1} \cap ... \cap F_{i_l}$ being the unique face that contains $p$ as an interior point and $G_f$ is the subgroup of $\mathbb{Z}_2 ^k$ generated by $\lambda(F_{i_1})$,..., $\lambda(F_{i_l})$ (if $p\in \mathrm{int}(\mathcal{P})$, we put $G_p=\{0\})$.

More practically, this means that $(\mathcal{P} \times \mathbb{Z}_2 ^k)/_{\sim}$ can be obtained from $\mathcal{P} \times \mathbb{Z}_2 ^k$ by identifying distinct copies $F\times \{g\}$ and $F\times \{h\}$ through the identity on $F \in \mathcal{F}$ whenever $g-h =\lambda(F)$. It is clear then that $\mathcal{M}_\lambda$ can be tessellated by $2^k$ copies of $\mathcal{P}$.
\end{rem}

We say that two colourings $\lambda_1,\lambda_2:\mathcal{F}\to V$ are \textit{equivalent} if there is a symmetry $g\in \, \mathrm{Sym}(\mathcal{P})$, which acts on $\mathcal{F}$ as a permutation; and an isomorphism $\varphi \in GL(V)$ such that $\lambda_2=\varphi\circ\lambda_1\circ g$. Clearly, equivalent colourings produce isometric manifolds.

\cite[Lemma 2.4]{KMT} also gives us the following criterion for orientability:

\begin{prop}\label{orientability}
$\mathcal{M}_\lambda$ is orientable if and only if $\lambda$ is equivalent to a colouring that assigns to each facet a colour with an odd number of entries 1.
\end{prop}

\begin{ex}[$2n$-gons]\label{example}
The lowest-rank colouring $\lambda$ on a right-angled $2n$-gon $P$ is given by associating alternately the canonical vectors $e_1$, $e_2$ to its edges. Since the polygon has Euler characteristic $\chi(P)=1-\frac{n}{2}$, we have from Proposition \ref{DJ} that $\chi(\mathcal{M}_\lambda)=4-2n$, that is, $\mathcal{M}_\lambda$ is an orientable surface of genus $g=n-1$. It is not hard to see that this is the only orientable rank $2$ colouring for $P$ up to equivalence.
\end{ex}

\subsection{Induced Colourings and Geodesic Embeddings}

Now, a facet $F_0\in \mathcal{F}$ of a right-angled, compact polytope $\mathcal{P} \subset \mathbb{H}^n$ is also a right-angled, compact, convex polytope in the hypersurface $\mathcal{H}\cong \mathbb{H}^{n-1}$ such that $\mathcal{H}\cap \mathcal{P}=F_0$. A $V$-colouring $\lambda$ of $\mathcal{P}$ then naturally induces a $W$-colouring $\mu$ of $F_0$, with $W\subset V$ generated by $\{\lambda(F)\mid F \, \text{adjacent to} \, F_0\}$: it suffices to define $\mu$ by $\mu(F_0\cap F)=\lambda(F)$, where $F$ runs over $\mathcal{F}(F_0)=\{F \in \mathcal{F}\mid F \, \text{is adjacent to} \, F_0\}$.

We have the following:

\begin{prop}[\cite{DJ91},\cite{KS}]\label{induced colouring}
The manifold $\mathcal{M}_\mu$ is contained (in a non-canonical way) in $\mathcal{M}_\lambda$ as a totally geodesic hypersurface so that the cover $\pi:\mathcal{M}_\lambda \to \mathcal{P}$ restricts to a cover $\mathcal{M}_\mu \to F_0$. Moreover, the number of connected components of $\pi^{-1}(F_0)$ is equal to $2^{k-s-1}$, where $k$ and $s$ are the ranks of $\lambda$ and $\mu$, respectively. Each connected component is a copy of $\mathcal{M}_\mu$ geodesically embedded in $M_\lambda$.
\end{prop}

\begin{ex}[L\"obell polyhedra]\label{Lobell example}
The L\"obell polyhedron $R(2n)$, which admits a hyperbolic, right-angled realization for $n\ge 3$, has a lowest-rank proper $\mathbb{Z}_2^3$-colouring $\Lambda_g$ described in \cite[Proposition 3.6]{Vesnin}, which assigns $e_1$ to one $P_{2n}$ face, $e_2$, $e_3$ alternately to its adjacent pentagonal faces, $e_3$ to the other $P_{2n}$ face and $e_1$, $(e_1+e_2+e_3)$ alternately to its adjacent faces (see Figure \ref{Lobell colouring}). The colouring induced by $\Lambda$ in one of the $P_{2n}$ faces, then, is equivalent to the colouring $\lambda$ of  $P_{2n}$ described in Example \ref{example}. By Proposition \ref{induced colouring}, it follows that $\mathcal{M}_{\lambda}$ embeds geodesically in $\mathcal{M}_{\Lambda}$, which is a closed, hyperbolic $3$-manifold, and also orientable by Proposition \ref{orientability}). 
\end{ex}

\begin{figure}[ht]
\includegraphics[scale=0.4]{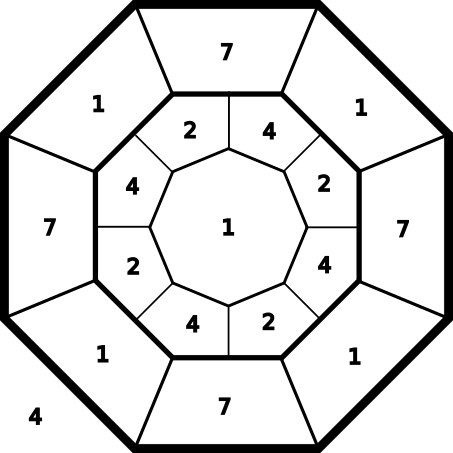}
\centering
\caption{Colouring of the polyhedron $R(8)$ as described in Example \ref{Lobell example}. Here the boundary of the polyhedron is seen in $\mathbb{R}^2 \cup \{\infty\}$ and the colours are represented in binary notation, associating the vector $(a,b,c)$ to $a+2b+4c$.}
\label{Lobell colouring}
\end{figure}

\subsection{Admissible symmetries and Coloured Isometries}

Now, given a (not necessarily proper) colouring $\lambda$ of a right angled polytope $\mathcal{P}$, there is a natural group of symmetries of the associated orbifold $\mathcal{M}_{\lambda}$ called the \textit{coloured isometry group} of $\mathcal{M}_{\lambda}$. We briefly recall its definition as given in \cite[Section 2.1]{KS}:

\begin{defn}\label{coloured symmetry group}
Let be a $V$--colouring of $\mathcal{P}$. A symmetry $g\in \text{Sym}(\mathcal{P})$, which acts on $\mathcal{F}$ as a permutation, is \textit{admissible with respect to $\lambda$} if:
\begin{enumerate}
    \item it induces a permutation of the colors assigned to the facets, by sending each color $\lambda(F)$ to $\lambda\big(g(F)\big)$;
    \item such permutation is realised by an invertible linear automorphism $\phi \in GL(V)$. 
    \end{enumerate}
    
Equivalently, $g$ is admissible with respect to $\lambda$ if $\exists \, \phi \in GL(V)$ such that $\phi^{-1}\circ \lambda \circ g = \lambda$.
\end{defn}

Admissible symmetries, which we denote by $\mathrm{Adm}_{\lambda}(\mathcal{P})$, are easily seen to form a subgroup of the symmetry group $\text{Sym}(\mathcal{P})$, and there is a naturally defined homomorphism $\Phi_{\lambda}:\mathrm{Adm}_{\lambda}(\mathcal{P})\to GL(V)$ such that $\Phi_\lambda(g)^{-1} \circ \lambda \circ g = \lambda$ for every $g\in \text{Adm}_\lambda (\mathcal{P})$.

The \textit{coloured isometry group}  $\mathrm{Isom}_c(M_{\lambda})$ is then defined as the group of symmetries of $\mathcal{M}_{\lambda}$ which are lifts of admissible symmetries of $\mathcal{P}$ under the orbifold cover $M_\lambda \to \mathcal{P}$. 

\begin{prop}[\cite{KS}, Section 2.1]\label{coloured isometry group}
For any $V$-colouring $\lambda$ of $\mathcal{P}$, we have that $\mathrm{Isom}_c(\mathcal{M}_{\lambda})\cong \mathrm{Adm}_{\lambda}(\mathcal{P})\ltimes V$. Moreover, the action of $\mathrm{Adm}_{\lambda}(\mathcal{P})$ on $V$ is precisely the one induced by the homomorphism $\Phi_{\lambda}$.
\end{prop}

\begin{corollary}\label{orientation-preserving coloured isometries}
The orientation-preserving coloured isometry group $\mathrm{Isom}_c ^+(\mathcal{M}_{\lambda})$ is generated by the isometries $(r,e_i)$, where $r\in \mathrm{Adm}_\lambda(\mathcal{P})$ is orientation-reversing; $(\mathrm{id},e_i+e_j)$ with $i\neq j$; and $(s,0)$ for $s\in \mathrm{Adm}_\lambda(\mathcal{P})$ orientation-preserving.
\end{corollary}

\begin{proof}
Since $\mathrm{ker} \, \lambda\, \triangleleft\, \Gamma$, any reflection $r_F \in \Gamma$ lifts to an isometry $\big(\mathrm{id},\lambda(F)\big)$ in $\mathcal{M}_\lambda \cong (\mathcal{P}\times V)/_\sim$, and it follows that $(\mathrm{id},v)\in \mathrm{Isom}_c (\mathcal{M}_\lambda)$ for all $v\in V$. 
Moreover, the isometries $(\mathrm{id},e_i)$ are always orientation-reversing, and we can conclude with the statement. 
\end{proof}

\begin{corollary}\label{fixed point-free isometries}
If the action of $\Phi_\lambda(\varphi)$ in $V$ is trivial, an isometry $(\varphi,w)$ is fixed point-free if and only if $w \notin \bigcup_{x \in \mathrm{Fix}(\varphi|_{\partial \mathcal{P}})} G_x$.
\end{corollary}

\begin{proof}
By Remark \ref{orbifold definition}, it is not hard to see that $(\varphi,w)\in \mathrm{Adm}_\lambda (\mathcal{P}) \ltimes V \cong \mathrm{Isom}_c (\mathcal{M}_\lambda)$ fixes a point $(x,v)\in (\mathcal{P}\times V)/_\sim \cong \mathcal{M}_\lambda$ if, and only if, $$(x,v)\sim(\varphi,w)\cdot (x,v)=(\varphi(x),\Phi_\lambda(\varphi)(v)+w),$$ that is, if $\varphi(x)=x$ and $\big(\Phi_\lambda(\varphi)(v)+w-v\big) \in G_x$. 
In particular, it is enough to check this condition only on the points $x \in \mathrm{Fix}(\varphi|_{\partial \mathcal{P}})$ and to guarantee that $\Phi_\lambda(\varphi)(v)\neq w+ v$ for every $v\in V$. We thus have the following:
\end{proof}

\begin{rem}\label{conformal automorphisms}
We note that, if $\mathcal{M}_\lambda$ is an orientable hyperbolic surface, the group of conformal automorphisms and orientation-preserving isometries coincide, and we have that $$\mathrm{Aut}(\mathcal{M}_\lambda) \cong \mathrm{Isom}^+(\mathcal{M}_\lambda) > \mathrm{Isom}_c^+(\mathcal{M}_\lambda),$$ where the last inclusion might be surjective. 
\end{rem}


\section{The Accola-Maclachlan surfaces}\label{AM}

Let $P_g$ be the right-angled $(2g+2)$-gon and $\lambda_g$ the unique lowest-rank colouring of $P_g$. By Example \ref{example}, $\mathcal{M}_{\lambda_g}$ a surface of genus $g$. We have the following:

\begin{prop}\label{isomorphic AM group}
The group $\text{Isom}_c ^+ (\mathcal{M}_{\lambda_g})$ of orientation-preserving coloured isometries of $\mathcal{M}_{\lambda_g}$ is isomorphic to the group $AM_g$ from Equation \ref{AM group}.
\end{prop}

\begin{proof}
Every symmetry of $P_g$ is admissible with respect to $\lambda_g$, so we have that $\mathrm{Adm}_{\lambda_g}(P_g)= \mathrm{Sym}(P_g) \cong D_{2g+2}$, the dihedral group of order $4g+4$. By Proposition \ref{coloured isometry group}, we have that $$ \text{Isom}_c (\mathcal{M}_{\lambda_g}) \cong D_{2g+2} \ltimes \mathbb{Z}_2 ^2,$$ a group of order $16g+16$. The group $D_{2g+2}$ is generated by two reflections $r_1$, $r_2$, where $r_1$ fixes some edge of $P_g$ and $r_2$ fixes one vertex of that edge (see Figure \ref{Kulkarni figures}--left for an example). These reflections act on $\mathbb{Z}_2 ^2$ respectively as the identity and as the permutation of coordinates. On the other hand, $\mathbb{Z}_2 ^2$ is generated by the canonical vectors $e_1$, $e_2$.

Let $\text{id}$ be the identity element in $D_{2g+2}$ and $0$ in $\mathbb{Z}_2 ^2$. Let also the product in these groups be given by $\cdot$ (which will be omitted) and $+$, respectively; while the product $\cdot$ in $D_{2g+2} \ltimes \mathbb{Z}_2 ^2$ is the one given by Proposition \ref{coloured isometry group} and the identity element will be noted by $1$. Put $$x=(\text{id},e_1)\cdot (r_1,0) = (r_1,e_1),\, y=(r_2,e_1),\, z=(r_1,e_2),\, w=(r_2,e_2).$$ We have that $x^2=1$, $xy=(r_1 r_2,0)$ and $y^2=(\text{id},e_1+e_2)$. Moreover, it is not hard to see that $y^2 x = z$ and $y^3=w$. We also note that, since the action of $r_1 r_2$ fixes $e_1+e_2$, it follows that $xy=(r_1 r_2,0)$ and $y^2=(\text{id},e_1+e_2)=y^{-2}$ commute. Similarly, since the action of $r_1$ on $\mathbb{Z}_2 ^2$ is trivial, $x=(r_1,e_1)=x^{-1}$ and $y^2$ commute as well, or, equivalently, $(xy^2)^2=1$.

Let $a=xy$, $b=y$. By the remarks above and Corollary \ref{orientation-preserving coloured isometries}, these generate $\text{Isom}_c ^+ (\mathcal{M}_{\lambda_g})$. Since $r_1 r_2$ is the rotation of $P_g$ by an angle of $\frac{2\pi}{2g+2}$, we also have that $a^{2g+2}=1$. Moreover, $b^4=(y^2)^2=1$, $[a,b^2]=(ab)^2=1$. Thus, $\mathrm{Isom}_c ^+(\mathcal{M}_{\lambda_g})$ is a group of order $8g+8$ with the same generators and relations of the Accola-Maclachlan group $AM_g$, and it follows that these two groups are isomorphic.
\end{proof}





We now can prove the following:

\begin{theorem}\label{is AM}
The manifold $\mathcal{M}_{\lambda_g}$ is isometric to the Accola-Maclachlan surface $X_g$.
\end{theorem}

\begin{proof}
By Proposition \ref{isomorphic AM group}, we have that $\text{Isom}_c ^+ (\mathcal{M}_{\lambda_g})\cong AM_g$. 
It is also not hard to see that the quotient $\mathcal{M}_{\lambda_g}/_{\text{Isom}_c ^+ (\mathcal{M}_{\lambda_g})}$ is obtained by the identification of two triangles from the first barycentric subdivision of $P_g$ along their sides, which gives us an orbifold with signature $(0;2,4,2g+2)$. Since $\text{Aut}(\mathcal{M}_{\lambda_g})>\text{Isom}_c ^+ (\mathcal{M}_{\lambda_g})\cong AM_g$ by Remark \ref{conformal automorphisms}, it follows from the non-extendability of the $AM_g$ group in signature $(0;2,4,2g+2)$ that $\text{Aut}(\mathcal{M}_{\lambda_g})\cong AM_g$. By the uniqueness of the conformal automorphism group of the Accola-Maclachlan surface on that signature, we have that $\mathcal{M}_{\lambda_g}\cong X_g$.
\end{proof}

\begin{corollary}\label{AM embeds}
The Accola-Maclachlan surface $X_g$ embeds geodesically for every genus $g$.
\end{corollary}

\begin{proof}
Take the right-angled L\"obell polyhedron $R(2g+2)$ with colouring $\Lambda_g$ as described in Example \ref{Lobell example}. We then have that the induced colouring in one $P_g$ face is precisely $\lambda_g$. It follows that $X_g\cong \mathcal{M}_{\lambda_g}$ embeds geodesically in $\mathcal{M}_{\Lambda_g}$.
\end{proof}

\begin{corollary}
The Accola-Maclachlan surface $X_g$ bounds geometrically an orientable manifold for every genus $g$.
\end{corollary}

\begin{proof}
Take the 
orientation-reversing involution $(r_1,e_1+e_2) \in \mathrm{Isom}_c(\mathcal{M}_{\lambda_g})$, composed of the orientation-reversing reflection $r_1$ and the orientation-reversing translations $e_1,e_2$. Since 
$\Phi_{\lambda_g}(r_1)$
acts trivially in $\mathbb{Z}_2 ^2$ 
and $ \mathrm{Fix}(r_1)\cap (\partial \mathcal{P}) =\{x_1,x_2\}$ with $G_{x_1}=G_{x_2}=[e_1]$, it follows from Corollary \ref{fixed point-free isometries} that $(r_1,e_1+e_2)$ is fixed point-free. 

We then have that $X_g\cong \mathcal{M}_{\lambda_g}$ has a fixed point-free, orientation-reversing involution for every $g$.\footnote{This result was also obtained by \cite[Section 2, p. 640]{BBCGG}.}
By Lemma \ref{lemma-embed-bound} and Corollary \ref{AM embeds}, we can conclude that $X_g$ bounds an orientable, compact manifold for every genus.
\end{proof}


\section{The Wiman type II surfaces}\label{Wiman}

Now, let $P_g '$ be the $4g$-gon and $\lambda_g '$ its unique lowest-rank colouring. From Theorem \ref{is AM}, we have that $$\text{Aut}(\mathcal{M}_{\lambda_g'}) \cong \mathrm{Isom}_c ^+(\mathcal{M}_{\lambda_g'})\cong \langle a,b \mid a^{4g}=b^4=(ab)^2=[a,b^2]=1\rangle,$$
where $a=(r_1 r_2,0)$ and $b=(r_2,e_1)$ in $\mathrm{Isom}_c (\mathcal{M}_{\lambda_g'})\cong D_{4g} \ltimes \mathbb{Z}^2 _2$. For every $g\ge 2$, take $c_g=a^{2g}b^2=\big( (r_1 r_2)^{2g},e_1+e_2\big) \in \text{Isom}_c ^+ (\mathcal{M}_{\lambda_g'})$. We have the following:

\begin{theorem}\label{is Wiman}
The manifold $\mathcal{M}_{\lambda_g'}/_{\langle c_g\rangle}$ is isometric to the Wiman type II surface $Z_g$.
\end{theorem}

\begin{proof}
We have that $c_g^2=1$. Now, $(r_1 r_2)^{2g}$ is precisely the antipodal map in $P_g'$, which only fixes the interior point $0\in P_g'$. Since the action of $(r_1 r_2)^{2g}$ in $\mathbb{Z}_2 ^2$ is trivial, it follows from Corollary \ref{fixed point-free isometries} that $c_g$ is fixed point-free isometry of $\mathcal{M}_{\lambda_g'}$.

Therefore, the orbifold $S_g\cong \mathcal{M}_{\lambda_g'}/_{\langle c_g\rangle}$ is an orientable, connected, closed manifold and it has Euler characteristic $\frac{4(1-g)}{2}=2-2g$, that is, is a surface of genus $g$. Also, because both $a$ and $b$ commute with $c_g$, we have that the normalizer of $c_g$ in $\mathrm{Aut}(\mathcal{M}_{\lambda_g'})$ is precisely the whole group. By Remark \ref{conformal automorphisms} and Lemma \ref{quotient lemma}, we have that $$\mathrm{Aut}\big(\mathcal{M}_{\lambda_g'}/_{\langle c_g \rangle}\big) > N_{\mathrm{Aut}(\mathcal{M}_{\lambda_g '})}(c_g)/_{\langle c_g \rangle} \cong \langle a,b\mid a^{4g}=b^4=(ab)^2=1,b^2=a^{2g}\rangle,$$
and it follows that $a$ descends into a conformal automorphism of $S_g$ of order $4g$. We thus conclude that $S_g$ is isometric to the Wiman type II surface for $g\neq 3$. 

We also note that $S_g/_{\langle a \rangle} \cong \mathcal{M}_{\lambda_g'}/_{\langle a, c_g\rangle}= \mathcal{M}_{\lambda_g'}/_{\langle a, b^2\rangle}$, and it is not hard to see from Remark \ref{orbifold definition} that this is an orbifold with signature $(0;2,4g,4g)$. It then follows that $S_g$ is also isometric to the Wiman type II surface for $g=3$.
\end{proof}

\begin{rem}\label{Wiman automorphism}
We note that $$\langle a,b\mid a^{4g}=b^4=(ab)^2=1,b^2=a^{2g}\rangle \cong \langle a,b' \mid a^{4g}=(b')^2=1,(b')^{-1}a(b')=a^{2g-1}\rangle,$$ with $b'=ab$. By \cite[Section 5.2]{Weaver}, this group is precisely $\mathrm{Aut}(Z_g)$, except in $g=2$. It follows that $\mathrm{Aut}(Z_g)\cong \mathrm{Isom}_c ^+(\mathcal{M}_{\lambda_g '})/_{\langle c_g \rangle}$ for every $g>2$.
\end{rem}

\begin{corollary}\label{Wiman embeds}
The Wiman type II surface $Z_g$ embeds geodesically for every genus $g$.    
\end{corollary}

\begin{proof}
Take the L\"obell polytope $R(4g)$ as described in Example \ref{Lobell example}. This polytope has symmetry group $D_{4g} \rtimes \mathbb{Z}_2$, where $D_{4g}$ is the symmetry group of the $4g$-gonal faces. Take also the colouring $\Lambda_g '$ described in Example \ref{Lobell example}. We have that the rotation $r$ of angle $\pi$ around the axis that crosses the center of the $4g$-gonal faces preserves each colour, and thus is an admissible symmetry for $\Lambda_g'$ which acts trivially in $\mathbb{Z}_2 ^3$. Moreover, it restricts to the symmetry $(r_1 r_2)^{2g}$ in each $4g$-gonal face.

Take, then, the isometry $S=(r,e_2+e_3)\in \mathrm{Isom}_c ^+(\mathcal{M}_{\Lambda_g '})$, which is composed of the orientation-preserving rotation $r$ and two orientation-reversing translations $e_2$ and $e_3$. We have that $S^2=\mathrm{Id}$. Moreover, the fixed points of $r$ are the rotation axis in $R(4g)$, and we have that $\mathrm{Fix}(r)\cap (\partial \mathcal{P})=\{x_1,x_2\}$, where $x_1,x_2$ lie in the center of the $4g$-gonal faces coloured with $e_1, e_3$ respectively. Since $G_{x_1}=\{0,e_1\}$ and  $G_{x_1}=\{0,e_3\}$, it follows from Corollary \ref{fixed point-free isometries} that $S$ is a fixed point-free, orientation-preserving isometry of $\mathcal{M}_{\Lambda_g '}$.

Now, by Example \ref{Lobell example}, we have that $\big(P_g ' \times \{0\} \times \mathbb{Z}_2 ^2\big)/_{\sim_{\lambda_g '}}$ is an embedded copy of $\mathcal{M}_{\lambda_g '}$ in $\mathcal{M}_{\Lambda_g '} \cong \big(R(4g)\times \mathbb{Z}_2 ^3\big)/_{\sim_{\Lambda_g '}}$, obtained from the induced colouring on the $4g$-gonal face coloured with $e_1$. Moreover, it is not hard to see that the isometry $S$, restricted to $\big(P_g ' \times \{0\} \times \mathbb{Z}_2 ^2\big)/_\sim$, acts precisely as $c_g$ on this copy of $\mathcal{M}_{\lambda_g '}$. It follows from Theorem \ref{is Wiman} that $Z_g\cong \mathcal{M}_{\lambda_g '}/_{\langle c_g \rangle}$ embeds geodesically in the compact, orientable, hyperbolic $3$-manifold $\mathcal{M}_{\Lambda_g'}/_{\langle S \rangle}$.
\end{proof}

We still need one result:

\begin{corollary}\label{Wiman involutions}
For $g\ge 3$, the Wiman type II surface $Z_g$ admits an orientation-reversion, fixed point-free involution if, and only if, $g$ is odd.
\end{corollary}

\begin{proof}
Since each of the isometries $(r_1,0), (r_2,0), (\mathrm{id},e_1), (\mathrm{id},e_2)\in \mathrm{Isom}_c(\mathcal{M}_{\lambda_g '})$ commute with $c_g$, we have that $$\mathrm{Isom}\big(\mathcal{M}_{\lambda_g '}/_{\langle c_g \rangle}\big) > \mathrm{Isom}_c\big(\mathcal{M}_{\lambda_g '}/_{\langle c_g \rangle}\big) > N_{\mathrm{Isom}_c(\mathcal{M}_{\lambda_g '})}(c_g)/_{\langle c_g \rangle} = \mathrm{Isom}_c(\mathcal{M}_{\lambda_g '})/_{\langle c_g \rangle},$$ and $\mathrm{Isom}_c(\mathcal{M}_{\lambda_g '})/_{\langle c_g \rangle}$ contains $\mathrm{Isom}_c ^+(\mathcal{M}_{\lambda_g '})/_{\langle c_g \rangle}$ as an index--$2$ subgroup. However, by Remark \ref{Wiman automorphism}, we have that $\mathrm{Isom}^+\big(\mathcal{M}_{\lambda_g '}/_{\langle c_g \rangle}\big)\cong \mathrm{Isom}_c ^+(\mathcal{M}_{\lambda_g '})/_{\langle c_g \rangle}$ for $g\ge 3$, from which it follows that $\mathrm{Isom}\big(\mathcal{M}_{\lambda_g '}/_{\langle c_g \rangle}\big) \cong \mathrm{Isom}_c(\mathcal{M}_{\lambda_g '})/_{\langle c_g \rangle}$ since the groups have the same order.

Now, take any isometry $\varphi =(s,v) \in \mathrm{Sym(\mathcal{P})}\ltimes V \cong \mathrm{Isom}_c(\mathcal{M}_{\lambda_g '})$, which descents into an isometry $\tilde{\varphi}\in \mathrm{Isom}_c\big(\mathcal{M}_{\lambda_g '}/_{\langle c_g \rangle}\big)$. We have the following:
\begin{enumerate}
    \item $\tilde{\varphi}$ is orientation-reversing if $\varphi$ is orientation-reversing;
    \item $\tilde{\varphi}$ is an involution if $\varphi^2 \in \langle c_g \rangle$;
    \item $\tilde{\varphi}$ fixed point-free if $\langle \varphi, c_g \rangle$ acts freely on $\mathcal{M}_{\lambda_g '}$. 
\end{enumerate}
If $\varphi^2=1$, we must have that $s^2=1$ and $v \in \mathrm{Fix}\big(\Phi_{\lambda_g '}(s)\big)$. Up to conjugation, we have then that $s$ must be either $r_1$, $r_2$ or $(r_1r_2)^{2g}$. If $s=r_1$, we have that $\mathrm{Fix}\big(\Phi_{\lambda_g '}(s)\big)=\mathbb{Z}_2 ^2$ and $\mathrm{Fix}(s)\cap (\partial \mathcal{P})=\{x_1,x_2\} $ with $G_{x_1}=G_{x_2}=[e_1]$. In this case, by Corollary \ref{fixed point-free isometries}, $(s,v)$ is fixed point-free only if $v\in\{e_2,e_1+e_2\}$. Since we want $\varphi=(r,s)$ to be orientation-reversing, we must then have $v=e_1+e_2$, from which it follows that $\varphi$ is a fixed point-free orientation-reversing involution. But then we would have that $\varphi c_g=(r_1 (r_1r_2)^{2g},0)$ is not fixed point-free. The other two cases $s=r_2$, $s=(r_1 r_2)^{2g}$ are excluded analogously.

Suppose then that $\varphi^2=c_g=\big((r_1r_2)^{2g},e_1+e_2\big)$. It follows, in this case, that $s^2=(r_1r_2)^{2g}$ and $v+\Phi_{\lambda_g '}(s)(v)=e_1+e_2$, which implies that $v \notin \text{Fix}\big(\Phi_{\lambda_g '}(s)\big)$. Since $s\in D_{4g}$, we must have that $s=(r_1r_2)^{\pm g}$. On the other hand, the possible actions of $\Phi_{\lambda_g '}(s)$ on $\mathbb{Z}_2 ^2$ are either trivial or permuting the coordinates, so we must have $v\in \{e_1,e_2\}$ and $\Phi_{\lambda_g '}(s)$ permuting the coordinates. For $s=(r_1r_2)^{\pm g}$, this is true if, and only if, $g$ is odd. In this case, we have that $\varphi$ is orientation-reversing and $\langle \varphi, c_g\rangle=\langle \varphi \rangle$ with both $\varphi$, $\varphi^3=\varphi^{-1}$ acting freely, and it follows that $\tilde{\varphi}$ is an orientation-reversing, fixed point-free involution of $\mathcal{M}_{\lambda_g '}/_{\langle c_g \rangle} \cong Z_g$.
%
\end{proof}


We can now finally prove this:

\begin{corollary}
The Wiman type II surface $Y_g$ bounds geometrically an orientable, closed manifold for every odd genus $g$.
\end{corollary}

\begin{proof}
This follows directly from Lemma \ref{lemma-embed-bound} and Corollaries \ref{Wiman embeds}, \ref{Wiman involutions}.
\end{proof}


\section{The Kulkarni Surfaces}\label{Kulkarni}

Now, let $g \equiv 3 \,(\text{mod} \, 4)$ and $P_g$ be the $(2g+2)$-gon. Let $\lambda_g ''$ be the colouring on $P_g$ given by Figure \ref{Kulkarni figures}--left, where each side is labelled with the vectors $e_1$, $e_2$, $e_3$ and $(e_1+e_2+e_3)$ in clockwise order. Since each adjacent side has distinct vectors, this is a proper colouring, and the surface $\mathcal{M}_{\lambda_g ''}$ is orientable by Proposition \ref{orientability}. Moreover, $\chi(\mathcal{M}_{\lambda_g ''})=8 \cdot \frac{2-2g}{4}=4-4g$, that is, $\mathcal{M}_{\lambda_g ''}$ is a surface of genus $2g-1$.

\begin{figure}[ht]
\centering
\begin{subfigure}[h]{0.4\textwidth}
\includegraphics[width=\textwidth]{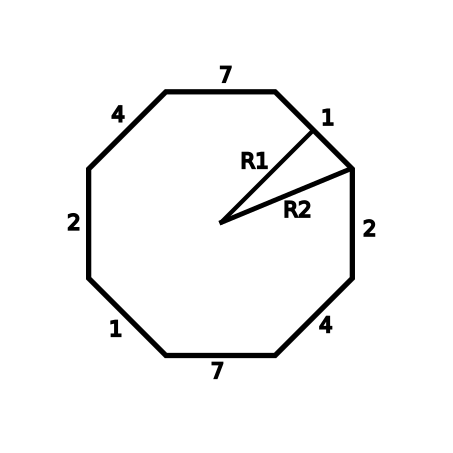}
\end{subfigure}
\begin{subfigure}[h]{0.4\textwidth}
\includegraphics[width=\textwidth]{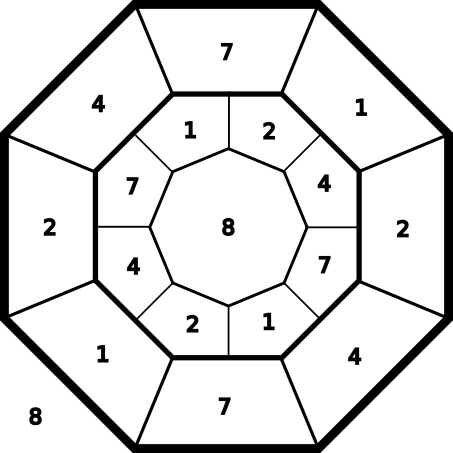}
\end{subfigure}
\caption{Respective colouring of the $(2g+2)$-gon and of the Lobell polyhedron $R(2g+2)$, as described in Section \ref{Kulkarni}, where colours are represented in binary notation.}
\label{Kulkarni figures}
\end{figure}

Let $r_1,r_2$ be the reflections shown in Figure \ref{Kulkarni figures}--left, which generate the group of symmetries $\mathrm{Sym}(P_g)\cong D_{2g+2}$. Due to the fact that $2g+2 \equiv 0 \,(\text{mod} \,8)$, we have that both these symmetries are admissible, with induced actions on $\mathbb{Z}_2 ^3$ given by 
\begin{equation*}
\Phi_{\lambda_g ''} (r_1)= \begin{pmatrix}
1 & 1 & 0 \\
0 & 1 & 0 \\
0 & 1 & 1 \\
\end{pmatrix}, \,\, 
\Phi_{\lambda_g ''} (r_2)= \begin{pmatrix}
0 & 1 & 1 \\
1 & 0 & 1 \\
0 & 0 & 1 \\
\end{pmatrix}, %
\end{equation*}
respectively. It follows that $\mathrm{Adm}_{\lambda_g ''}(P_g)= \mathrm{Sym}(P_g) \cong D_{2g+2}$, and by Proposition \ref{coloured isometry group}, we have that $ \text{Isom}_c (\mathcal{M}_{\lambda_g ''}) \cong D_{2g+2} \ltimes \mathbb{Z}_2 ^3$, with the product given by the action $\Phi_{\lambda_g ''}$.

Let $d_g =\big( (r_1 r_2)^{g+1},e_1+e_3\big) \in \text{Isom}_c ^+ (\mathcal{M}_{\lambda_g ''})$. We have the following:

\begin{theorem}\label{is Kulkarni}
The surface $\mathcal{M}_{\lambda_g ''}/_{\langle d_g \rangle}$ is isometric to the Kulkarni surface $Y_g$ for every $g \equiv 3 \,(\text{mod} \,4)$.
\end{theorem}

\begin{proof}
Put $x=(r_1,e_1)$, $y=(r_1,e_2)$, $z=(r_2,e_1)$ and $w=(r_2,e_3)$. Since $xy^2=(r_1,e_3)$, $z^3=(r_2,e_2)$, $y=z^2x^{-1}$ and $w=y^2 z$, we have by Corollary \ref{orientation-preserving coloured isometries} that $x,z$ generate the orientation-preserving coloured isometry group.

Now, since $g + 1\equiv 0 \,(\text{mod} \, 4)$, we have that $\Phi_{\lambda_g ''} \big((r_1 r_2)^{g+1}\big)$ acts trivially in $\mathbb{Z}_2 ^3$, and it follows from Corollary \ref{fixed point-free isometries} that $d_g$ is fixed point-free. Also, $d_g^2=1$, and since $e_1+e_3$ is fixed by both actions $\Phi_{\lambda_g ''}(r_1)$, $\Phi_{\lambda_g ''}(r_2)$, it follows that $d_g$ commutes with $x$ and $z$ and thus with every element of $\text{Isom}_c ^+ (\mathcal{M}_{\lambda_g ''})$. Equivalently, $N(d_g)=\text{Isom}_c ^+ (\mathcal{M}_{\lambda_g ''})$, where the normalizer is taken inside $\text{Isom}_c ^+ (\mathcal{M}_{\lambda_g ''})$.

Put $a=xz=(r_1 r_2,0)$, $b=z^3=z^{-1}=(r_2,e_2)$. We have that $ a^{2g+2}=b^4=1$, $(ab)^2=x^2=1$ and $b^2 a b^2 = a^{g+2} d_g$. It follows that $N(d_g)/_{\langle d_g\rangle}$ is a group of order $8g+8$ with the same presentation as the Kulkarni group $K_g$, that is, $N(d_g)/_{\langle d_g\rangle}\cong K_g$. Also, $$\big(\mathcal{M}_{\lambda_g ''}/_{\langle d_g \rangle}\big)/_{\big(N(d_g)/_{\langle d_g \rangle}\big)} \cong \mathcal{M}_{\lambda_g ''}/_{\text{Isom}_c ^+ (\mathcal{M}_{\lambda_g ''})},$$ which is an orbifold obtained similiarly to the one in the proof of Theorem \ref{is AM}, and thus also has signature $(0;2,4,2g+2)$. Moreover, by Remark \ref{conformal automorphisms} and Lemma \ref{quotient lemma}, $$\mathrm{Aut}\big(\mathcal{M}_{\lambda_g ''}/_{\langle d_g \rangle}\big) > N_{\mathrm{Aut}(\mathcal{M}_{\lambda_g ''})}(d_g)/_{\langle d_g \rangle} > N_{\text{Isom}_c ^+(\mathcal{M}_{\lambda_g ''})}(d_g)/_{\langle d_g\rangle} = N(d_g)/_{\langle d_g \rangle} \cong K_g,$$ where the second inclusion comes from the fact that $\mathrm{Isom}_c ^+ (\mathcal{M}_{\lambda_g ''}) < \mathrm{Aut} (\mathcal{M}_{\lambda_g ''}) $. We can thus conclude that $\mathcal{M}_{\lambda_g ''}/_{\langle d_g\rangle} \cong Y_g$, the Kulkarni surface of genus $g$.
\end{proof}

\begin{rem}\label{common double-cover}
We note that, if we would take the isometry $e=(\mathrm{id},e_1+e_3)$, we would have that $\mathcal{M}_{\lambda_g ''}/_{\langle e \rangle} \cong \mathcal{M}_{\lambda_g} \cong X_g$. Indeed, the isotropy groups $G_x$ for the vertices $x\in P_g '$ with colouring $\lambda_g ''$ are one of the following four subspaces of $\mathbb{Z}_2 ^3$: $$[e_1,e_2], [e_2,e_3], [e_3,e_1+e_2+e_3]=[e_3,e_1+e_2], [e_1+e_2+e_3,e_1]=[e_1,e_2+e_3].$$ Since $(e_1+e_3)$ does not belong to any of these subspaces, it follows from Corollary \ref{fixed point-free isometries} that $e\in \mathrm{Isom}_c ^+(\mathcal{M}_{\lambda_g ''})$ is fixed point-free. Moreover, $\mathcal{M}_{\lambda_g ''}/_{\langle e \rangle}\cong \mathcal{M}_{(p \, \circ \lambda_g '')}$, where $p:\mathbb{Z}_2 ^3 \to \mathbb{Z}_2 ^2$ is the projection that maps $e_1,e_3\mapsto e_1$ and $e_2\mapsto e_2$.
We conclude with the fact that $p\circ \lambda_g ''=\lambda_g$.
\end{rem}

\begin{corollary}\label{Kulkarni embeds}
The Kulkarni surface $Y_g$ embeds geodesically for every $g \equiv 3 \,(\text{mod} \,4)$.
\end{corollary}

\begin{proof}
Take the rank--$4$ colouring $\Lambda_g ''$ of the L\"obell polyhedron $R(2g+2)$, as given in Figure \ref{Kulkarni figures}--right, which assigns $e_4$ to the $(2g+2)$-gonal faces; $e_1$, $e_2$, $e_3$, $(e_1+e_2+e_3)$ cyclically in the pentagonal faces adjacent to one $(2g+2)$-gonal face; and the same cyclical assignment rotated by two pentagonal faces in the remaining pentagonal faces. This assignment is always possible for $2g+2 \equiv 0 \,(\text{mod} \,8)$, and it gives a proper colouring since every three faces meeting at a vertex have linearly independent colours. It follows from Propositions \ref{DJ}, \ref{orientability} that $\mathcal{M}_{\Lambda_g ''}$ is a closed, orientable, hyperbolic $3$-manifold.

Now, as in the proof of Corollary \ref{Wiman embeds}, take the rotation $r'$ of angle $\pi$ around the axis that crosses the center of the $(2g+2)$-gonal. This rotation fixes each colour, and thus is an admissible symmetry for $\Lambda_g''$ which acts trivially in $\mathbb{Z}_2 ^4$. Moreover, it restricts to the symmetry $(r_1 r_2)^{g+1}$ in each $(2g+2)$-gonal face.

Take, also, the isometry $S'=(r',e_1+e_3)\in \mathrm{Isom}_c ^+(\mathcal{M}_{\Lambda_g ''})$. Similarly to the proof of Corollary \ref{Wiman embeds}, we can show that $S'$ is fixed point-free, orientation-preserving isometry of $\mathcal{M}_{\Lambda_g ''}$. It's also not hard to see that the induced colouring in each $P_g$ face is $\lambda_g ''$. Moreover, $S'$ acts on each $(2g+2)$-gonal face precisely as $d_g$ acts on $P_g$. Since $\mathcal{M}_{\lambda_g ''} $ embeds geodesically in $\mathcal{M}_{\Lambda_g ''} $ by Proposition \ref{induced colouring}, it follows from Theorem \ref{is Kulkarni} that $Y_g\cong \mathcal{M}_{\lambda_g ''}/_{\langle d_g \rangle}$ embeds geodesically in the compact, orientable, hyperbolic $3$-manifold $\mathcal{M}_{\Lambda_g''}/_{\langle S' \rangle}$.
\end{proof}

\begin{corollary}
The Kulkarni surface $Y_g$ bounds geometrically an orientable, closed manifold for every $g \equiv 3 \,(\text{mod} \,8)$. 
\end{corollary}

\begin{proof}
Similarly to the proof of Corollary \ref{Wiman involutions}, we have that $d_g$ commutes with every element of $\mathrm{Isom}_c(\mathcal{M}_{\lambda_g ''})$ and $\mathrm{Isom}_c(\mathcal{M}_{\lambda_g ''})/_{\langle d_g \rangle} < \mathrm{Isom}\big(\mathcal{M}_{\lambda_g ''}/_{\langle d_g \rangle}\big)$, where the inclusion is surjective for $g\neq 3$ by the non-extendability of the action of the Kulkarni group $K_g$ on $Y_g$, which implies that both groups must have the same order.

As in the proof of Corollary \ref{Wiman involutions}, we have that an isometry $\varphi=(s,v)\in \mathrm{Isom}_c(\mathcal{M}_{\lambda_g ''})$ will descend to a fixed point-free, orientation-reversing involution $\tilde{\varphi}\in \mathrm{Isom}\big(\mathcal{M}_{\lambda_g ''}/_{\langle d_g\rangle}\big)$ if, and only if, $\varphi$ is orientation-reversing, $\varphi^2 \in \langle d_g \rangle$ and $\langle \varphi, d_g\rangle$ acts freely on $\mathcal{M}_{\lambda_g ''}$. As in the proof of of Corollary \ref{Wiman involutions}, we can show that there is no such $\varphi$ in the case $\varphi^2=1$.

Take then $\varphi^2=d_g$. As in the proof of Corollary \ref{Wiman involutions}, we must have then that $s=(r_1 r_2)^{\pm\frac{g+1}{2}}$ and $v+\Phi_{\lambda_g ''}(s)(v)=e_1+e_3$, from which it follows that $v\notin \mathrm{Fix}\big(\Phi_{\lambda_g ''}(s)\big)$. If $g \equiv 7 \,(\text{mod} \,8)$, however, we have that $\Phi_{\lambda_g ''}(s)=\mathrm{Id}$ and thus there is no such $\varphi$ with $\varphi^2=d_g$. If $g \equiv 3 \,(\text{mod} \,8)$, on the other hand, $\Phi_{\lambda_g ''}(s)=\Phi_{\lambda_g ''}(r_1)$ and both $\varphi$, $\varphi d_g=\varphi^3=\varphi^{-1}$ are fixed point-free for any choice of $v\in \{e_1,e_2,e_3,e_1+e_2+e_3\}$, and orientation-reversing as well.

We thus have that the Kulkarni surface $Y_g\cong \mathcal{M}_{\lambda_g ''}/_{\langle d_g\rangle}$ admits an orientation-reversing, fixed point-free involution if, and only if, $g \equiv 3 \,(\text{mod} \,8)$.\footnote{This result was also obtained by \cite[Section 2, p. 640]{BBCGG}.} We can then conclude with Lemma \ref{lemma-embed-bound} and Corollary \ref{Kulkarni embeds}.
\end{proof}

\end{document}